\documentclass[12pt, a4paper]{article}

\usepackage{epsf}
\usepackage{psfrag}
\usepackage{amsmath}
\usepackage{amssymb}
\usepackage{lscape}
\usepackage{latexsym}
\usepackage[usenames]{color}
\usepackage[mathcal]{eucal}
\usepackage{stmaryrd}
\usepackage{textcomp}
\usepackage{setspace} 
\usepackage{graphicx} 
\usepackage{pgfplots} 
\usepackage{epsfig}   
\usepackage{epstopdf} 
\usepackage{appendix}
\usepackage{afterpage}
\usepackage{float} 
\usepackage[misc,geometry]{ifsym} 
\usepackage[font=small,skip=0pt]{caption}

\newtheorem{proposition}{Proposition}[section]

\newtheorem{remark}{Remark}[section]

\newtheorem{algorithm}{Algorithm}[section]

\newcommand{\ds}{\displaystyle}
\newcommand{\dZ}{{\cal Z \kern -0.7em Z}}
\newcommand{\dC}{{\rm\hbox{C \kern-0.8em\raise0.2ex\hbox{\vrule height5.4pt width0.7pt}}}}
\newcommand{\dQ}{{\rm\hbox{Q \kern-0.85em\raise0.25ex\hbox{\vrule height5.4pt width0.7pt}}}}

\newcommand{\proofbox}{\hspace{\fill}{$\Box$}}

\newenvironment{proof}{Proof.}{\proofbox}

\newcommand\old[1]{}
\newcommand{\beqa}{\begin{eqnarray*}}
\newcommand{\eeqa}{\end{eqnarray*}}

\title{\bf \Large{{A new mathematical model and algorithm for the optimal path to intercept a moving target}}}

\author{M. Akter  \thanks{Department of Mathematics, University of Chittagong, Chittagong-4331, Bangladesh,
   {\tt email: masuda@cu.ac.bd},}
\and
M. M. Rizvi \Letter \thanks{Centre for Smart Analytics (CSA), Institute of Innovation, Science and Sustainability, Federation University Australia, {\tt email:  m.rizvi@federation.edu.au}, (corresponding author)}
\and
M. Forkan \thanks{Department of Mathematics, University of Chittagong, Chittagong-4331, Bangladesh, 
{\tt email: forkan.math@cu.ac.bd},.}}
\begin{document}
%
\maketitle

\begin{abstract}
 \noindent This paper is concerned with determining the shortest path for a pursuer aiming to intercept a moving target travelling at a constant speed. To address this challenge, we introduce an  efficient mathematical model outlined as an optimal control problem. The proposed model is based on Dubin's path, where we concatenate two possible paths: a left-circular curve or a right-circular curve followed by a straight line. We develop and explore this model, providing a comprehensive geometric interpretation, and design an algorithm tailored to implement the proposed mathematical approach efficiently. Extensive numerical experiments involving diverse target positions highlight the strength of the model. The method exhibits a remarkably high convergence rate in finding solutions. We compare the proposed model and demonstrate its advantages through examples. For experiment purposes, we utilized the modelling software AMPL, employing a range of solvers to solve the problem. Subsequently, we simulated the obtained solutions using MATLAB, demonstrating the efficiency of the model in intercepting a moving target. The proposed model distinguishes itself by employing fewer parameters and making fewer assumptions, setting the model simplifies the complexities, and thus, makes it easier for experts to design optimal path plans. 
\end{abstract}


\textbf{AMS} subject classifications: 34K35 49M37, 90C30, 	90C39.

\section{Introduction}\label{sec1}
\noindent Path planning for intercepting a moving target has been a significant challenge, demanding efficient strategies for path planning to ensure successful interception. This has applications in various domains such as robotics, aerospace, and agricultural applications \cite{Wei2008}-\cite{Höffmann2024}.
The motivation behind the study of pursuer-target interception lies in real-world scenarios where one object needs to intersect with another.  Researchers are exploring advanced algorithms, mathematical models, and artificial intelligence methods to enhance the accuracy and speed of interception. Our study focus deeply into the essential realm of pursuer-target interception path planning. By exploring innovative mathematical models and employing advanced computational techniques, our aim is to contribute significantly in this area. We propose efficient ways to minimize the pursuer's path when intercepting a moving target. Our research endeavors not only towards theoretical advancements but also practical solutions that can be universally applied, addressing the evolving challenges associated with intercepting a moving target in various domains.

\noindent Dubins \cite{Dubins1957} introduced the shortest path for vehicles moving only forward and turning in unidirectional turns (either left or right) from specific starting to ending positions in the 2D plane. The author described paths as combinations of curves ($C$) and straight ($S$) lines, where the vehicle starts and ends with a turn, and in the middle, it could be straight or curved. There are six possible combinations of these paths, represented by $D=\{LSL, RSL, RSR, LSR, LRL, RLR\}$, where L and R mean left and right turns, respectively. Reeds and Shepp \cite{Reeds1990}  later expanded on Dubins work, allowing backward adaptability in these paths. Dubins results have since been further developed, incorporating geometric principles \cite{Ayala2014} and leveraging advanced theories in nonlinear optimal control \cite{Sussmann1991}-\cite{Chen2020}. Specifically, Pontryagin’s Maximum Principle (PMP) \cite{Pontryagin1962} has been applied, enhancing the understanding and applications of Dubins paths. These developments have expanded the scope of Dubins work, making it a pivotal foundation in the study of optimal vehicle paths. 

\noindent Numerous researchers have also been made an effort to develop a model to figure out the shortest distance for intercepting a moving target. Boissonnat and Bui's study \cite{Boissonnat1994} on the Relaxed Dubins Problem (RDP) focuses on finding the shortest Dubins path from a given configuration to a stationary target. Dubin's path problem from a fixed initial configuration to intercept a target have been extensively researched in several fields \cite{Madveev2011, Madveev2012}. When intercepting a moving target, researchers have also explored Dubins and relaxed Dubins optimal pathways \cite{Gopalan2016}-\cite{Buzikov2021}. Looker \cite{Looker2008} proposes a search algorithm to determine the shortest $CS$ type path for interception, assuming the pursuer behaves like a Dubins vehicle and the target moves at a constant speed. Meyer et al. \cite{Meyer2014} analyze problems related to minimum-time interception and interception at a predetermined time for a target moving in a straight line. Additionally, Meyer et al. \cite{Meyer2015} address finding the shortest path for a Dubins vehicle without a terminal angle constraint, referred to as the `relaxed Dubins' problem. In our research, we explore similar scenarios and attempt to reformulate various mathematical models for intercepting a moving target.\\

\noindent Kaya \cite{Kaya2017} studied the reformulation of the Dubins path based on the optimality principle, aiming to find a time-optimal, collision-free trajectory for the Dubins vehicle. The author extended this approach to compute Dubins interpolating curves, which are the shortest curvature-constrained paths through a given sequence of points \cite{Kaya2019}. Zheng \cite{Zheng2021} introduced a model focusing on the Minimum-Time Intercept Problem (MTIP). This problem involves guiding a Dubins vehicle from a specific location with a specified initial heading angle to intercept a moving target in the shortest time possible. Later, the author discussed the same scenario with prescribed impact angle constraints \cite{Zheng2022}.  Recently, Forkan et al. \cite{Forkan2022} introduced an innovative mathematical model to find a minimum path length or minimum time for touring a finite number of targets by unmanned aerial vehicle (UAV), both the UAV and targets are at the same altitude. The authors used the arc parameterization technique to determine the optimal path. We employ the same methodologies to develop the mathematical model for finding the optimal path of the pursuer’s pathway to intercept a moving target that travels through a straight line at a constant speed.
In recent times, Hou \cite{Hou2023} suggests a geometry-based impact time and angle control guiding rules without calculating the time-to-go that guarantees target captureability and zero terminal guidance command.

\noindent This paper proposes a mathematical model formulated as an optimal control problem, leveraging Dubin's path with a combination of circular and straight-line segments. The model's geometric interpretation is illustrated in Section \ref{geometry1}, and an efficient algorithm is developed to approximate the optimal path of the pursuer for intercepting a moving target. Extensive numerical experiments demonstrate the model's efficiency, high convergence rate, and advantages over existing methods. The proposed model is also suitable for minimizing the time it takes while a moving target is intercepted by a pursuer travelling at a constant speed. The model is expected to have the following useful features:
\begin{enumerate}
    \item [(i)]	When faced with a complex problem that involves numerous state and control variables, it's crucial for the method to provide approximations of optimal solutions. This is especially important in real-life applications where the problem-solving process can get increasingly intricate. The biggest hurdle in such scenarios is to develop an algorithm that can effectively and efficiently solve the problem which is the main motivation of this paper. 
    \item[(ii)] The method should construct an algorithm to generate solutions efficiently at a reasonable computational time. Computational time management in our analysis is important because solving a large problem can be very costly.
    \item[(iii)] The method should optimize the heading angles to determine the pursuer’s direction to generate the shortest path of interception. This is an essential feature since it guarantees that, given those heading angles, one may optimize the pursuer's path or time without considering each of the feasible paths separately. Therefore, our goal is to propose a mathematical model and increase its capacity for approximating the optimum trajectory in terms of length and time through the solution of the mathematical model.
\end{enumerate}

\noindent We are therefore interested in studying more about how the suggested model might construct the optimum path with minimal computational effort. The strengths and advantages of the proposed method with features (i)-(iii) are illustrated through the examples. The proposed model is compared with an existing model presented in \cite{Looker2008}, and the findings are demonstrated in Figures (\ref{PurTar1}-\ref{comparison}).

\noindent The rest of this paper comprises the following sections. In Section \ref{Formulation OCP}, we illustrate reformulation of Dubins problem as a time-optimal control problem and Pontryagin’s Maximum Principle. The geometrical interpretation of optimal path planning of a pursuer for intercepting a moving target is described in Section \ref{geometry1}.  In Section \ref{modelM}, the model for finding an optimal path of the pursuer to intercept a moving target is proposed. In Section \ref{algorithm_target}, we introduce an algorithm to implement the proposed model.  Numerical experiments are conducted and we provide the results and discussions based on numerical experiments in Section \ref{sec6}. The last section presents the conclusion of the paper.

\color{black}

\section{Formulation of Optimal Control Problem}\label{Formulation OCP}
In this section, the optimal control problem for the pursuer and moving target interception is formulated and control law is discussed using Pontryagin’s Maximum Principle. 

\noindent Let us consider the optimum path, denoted as $z(t):[t_0, t_f] \to \mathbb{R}^2$, starting from the initial position to the interception point. This path $z(t)$ measure the entire travel path for both the pursuer and the target until they meet each other.  We assume that the target is moving in a straight line with a constant speed and the pursuer does not face any obstacles during the motion. We also assume that the pursuer can not reverse its direction to intercept a target. The pursuer also has complete knowledge of the target's future trajectory and speed. 

\noindent We consider $(x(t),y(t))\in \mathbb{R}^2$ along $z(t)$ with $\dot x(t)= |\mathbf{V}| \cos \theta(t)$ and $\dot y(t)= |\mathbf{V}|\sin \theta (t)$, where $\theta(t)$ is the heading angle, which is measured from the $x$-axis through counterclockwise and $\theta(t)\in [0,2\pi]$, $|\mathbf{V}|$ is the speed, $R$ is the minimum turning radius of the pursuer.  

\noindent Let $t\in [t_0, t_f],$ where $t_f$ be the terminal time which is unknown and needs to be determined, and $t_0$ be the initial time. Therefore, the minimum path of the pursuer for intercepting a moving target is determined  by the arc length functional
\begin{equation}
\int_{t_0}^{t_f} \sqrt{\dot{x}^2+\dot{y}^2}\,dt = \int_{t_0}^{t_f} ||\dot{z}(t)||\,dt=|\mathbf{V}|(t_f-t_0) ,   
\end{equation}
 where, $||\dot{z}(t)||=|\mathbf{V}|,$ $\ds \dot{z}(t)=\frac{dz}{dt}$, and $||.||$ is the Euclidean norm.

\noindent Furthermore, the unit tangent vector is defined by $\ds T(t)=\frac{\dot{z}(t)}{||\dot{z}(t)||}=(\cos\theta(t),\sin\theta(t))$, and the curvature at each point of the path is calculated as  $\ds a=\frac{||\dot{T}(t)||}{||\dot{z}(t)||}=\frac{|\dot{\theta}(t)|}{|\mathbf{V}|}$.

\noindent The quantity $\ds \frac{\dot{\theta}(t)}{|\mathbf{V}|}$ can be positive or negative, and referred to as the signed curvature. If $\dot{\theta}(t) > 0$ then the pursuer moves in the counter-clockwise direction (left turns ($L$)) and if $\dot{\theta}(t) < 0$ then the pursuer moves in the clockwise direction (right turns  ($R$)),  and also $\dot{\theta}(t)=0$, if the pursuer moves in the straight line ($S$). Let $u(t) = \ds \frac{\dot{\theta}(t)}{|\mathbf{V}|}$ be a control variable.

\noindent Let us consider that heading angle of pursuer at the starting position $P(x_{P_0}, y_{P_0})$ is $\theta_{P_0}$, and the constant velocities of pursuer and target are $\bf{V_P}$ and $\bf{V_T}$ respectively.  

\noindent Hence the optimal control problem for the pursuer can be formulated as follows.
\begin{equation}\label{problem1}
\begin{array}{cl} \min & \ f =\ds \int_{t_0}^{t_f} |\mathbf{V_P}| \,dt =|\mathbf{V_P}| (t_f-t_0).  \\[4mm]
\mbox{subject to} 
\\[1mm]
&\  \dot{x}_P(t)=|\mathbf{V_P}| \cos\theta_P(t), \,x_P(t_0)= x_{P_0}, \\[2mm]
    &  \ \dot{y}_P(t)= |\mathbf{V_P}|\sin{\theta}_P(t),\, y_P(t_0)= y_{P_0} , \\[2mm]
   & \ \dot{\theta}_P(t) =u(t)|\mathbf{V_P}|, \, \theta_P(t_0)=\theta_{P_0},  \\[2mm]
 &  \ |u(t)|\leq a,
\end{array}\\
\end{equation}
\noindent where ($x_P(t)$, $y_P(t)$ $\theta_P(t)$) is represent the position along the path of pursuer.

\noindent Let's consider the target position is ($x_T(t)$, $y_T(t)$ $\theta_T$)  moving with constant speed $|\mathbf{V_T}|>0$ without making any turns or maneuvers. The direction of the target's movement is indicated by $\theta_T$, which belongs to  $[0, 2\pi]$. Since the target does not change its direction, the value of $\theta_T$ remains fixed throughout the engagement. Assuming the target's initial position is $(x_{T_0},y_{T_0})$, its equations of motion can be expressed as follows:

\begin{equation}\label{problemT}
\begin{array}{cl} 
   &\  \dot{x}_T(t)=|\mathbf{V_T}| \cos\theta_T, \,x_T(t_0)= x_{T_0}, \\[2mm]
    &  \ \dot{y}_T(t)= |\mathbf{V_T}|\sin\theta_T,\, y_T(t_0)= y_{T_0}. \\[2mm]
\end{array}\\
\end{equation}

\subsection{Pontryagin's Maximum Principle } \label{sec4a}

In this subsection, we will state the necessary conditions of optimality by Pontryagin's maximum principle \cite{Pontryagin1962} for Problem (\ref{problem1}) to obtain the optimal path for intercepting a target by a pursuer. 

\noindent Define the Hamiltonian function for Problem (\ref{problem1}) as follows. 

\begin{equation}\label{problemT1}
\begin{array}{cl} 
   \hspace{-58mm}  H(x_P(t),y_P(t),\theta_P(t),\mu_0,\mu_1(t),\mu_2(t),\mu_3(t),u(t))=  \\ [2mm]
     \mu_0 |\mathbf{V_P}| + \mu_1(t) |\mathbf{V_P}| \cos\theta_P(t)+ \mu_2(t) |\mathbf{V_P}| \sin\theta_P(t)+\mu_3(t) \,\ds u(t)\, \ds |\mathbf{V_P}|,  \\[2mm]
\end{array}\\
\end{equation}
\noindent where $\mu_0$ is a scalar parameter and $\mu_i:[t_0,t_f]\in \mathbb{R}, i=1,2,3, $ are the costate variables.\\
\noindent The costate variables are required to satisfy
\begin{equation}\label{PMP2}
    \dot{\mu}_{1}(t) = -H_{x_P}(t)=0,
\end{equation}
\begin{equation} \label{PMP3}
   \dot{\mu}_{2}(t) = -H_{y_P}(t)=0, 
\end{equation}
\begin{equation}\label{PMP4}
    \dot{\mu}_{3}(t) = -H_{\theta_P}(t) = \mu_1(t) \,|\mathbf{V_P}| \sin\theta_P(t)-\mu_2(t)\,|\mathbf{V_P}| \cos\theta_P(t). 
\end{equation}
\noindent In addition to (\ref{PMP2})-(\ref{PMP4}), one has $\dot{x}_P(t)=H_{\mu_1}(t)$, $\dot{y}_P(t)=H_{\mu_2}(t)$, $\dot{\theta}_P(t)=H_{\mu_3}(t).$ Note that from (\ref{PMP2}) and (\ref{PMP3}),  we have that $\mu_1(t)= \bar{\mu}_1$, $\mu_2(t)= \bar{\mu}_2$ for all $ t\in [t_0, t_f],$ where $\bar{\mu}_1$ and $\bar{\mu}_2$ are constants.

\noindent Moreover, the state differential equations  and other constraints given in (\ref{problem1}) and the adjoint differential equations (\ref{PMP2}) - (\ref{PMP4}), the following conditions hold:\\
\begin{equation}\label{PMP5}
  u(t)\in \underset{||w||\leq a}{arg\,min} \;H(x_P(t),y_P(t),\theta_P(t),\mu_0,\mu_1(t),\mu_2(t),\mu_3(t),w), 
\end{equation}
\begin{equation}\label{PMP6}
    H(t)=0.
\end{equation}
\noindent Thus, using (\ref{problemT1}), equation (\ref{PMP5}) can be written more simply as,
\begin{equation}\label{PMP7}
  u(t)\in \underset{||w||\leq a}{arg\,min} \;\;\mu_3(t)w, 
\end{equation}
\noindent and therefore, equation \eqref{PMP7} implies that the optimal controls $u(t)$ satisfy the following control law:\\
\begin{equation}\label{PMP8}
\hspace{20mm}
u(t) = 
\begin{cases} 
~~a,   &\text{if} \,\, \mu_3(t) < 0 \\
-a,  &\text{if}  \,\,\mu_3(t) > 0 \\
\mbox{undetermined}, &\text{if}\,\, \mu_3(t)= 0
\end{cases} 
\end{equation}
for all $t \in [t_0, t_f]$, where $\mu_3(t)$ acts as a function for switching between controls with two extremes $a$ and $-a$.  
\color{black}

\section {Interception of Pursuer and Moving Target: A Geometric Perspective} \label{geometry1}

\begin{figure}[hbt!]
	\begin{minipage}{100mm}
		\begin{center}
			\hspace*{15mm}
			\includegraphics[width=120mm]{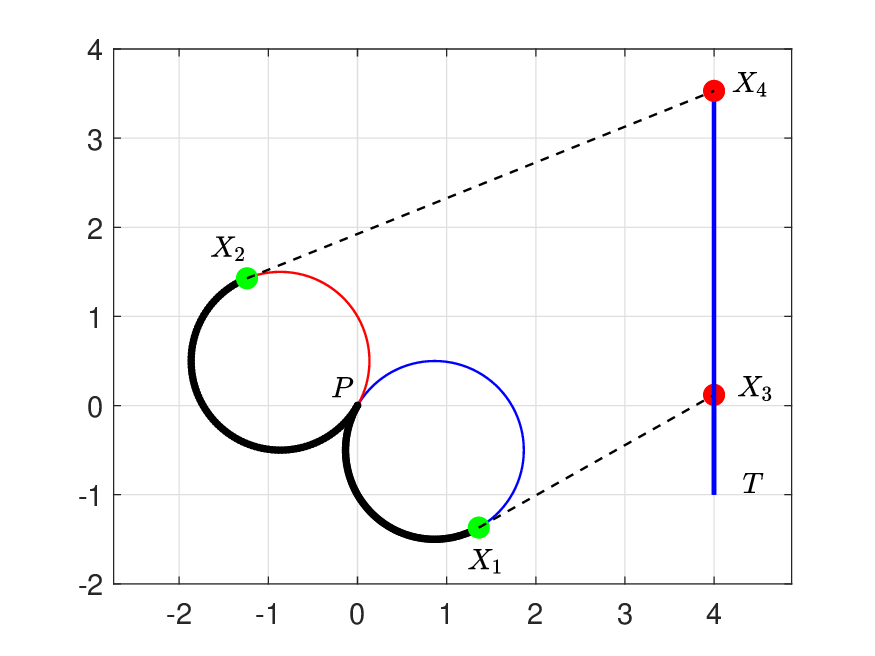} \\[3mm] \vspace{-12mm}
		 \vspace{4mm}
		\end{center}
	\end{minipage} \vspace{4mm}
	\caption{ \small{ Feasible optimal paths of interception of a moving target and a pursuer}.}
	\label{Figopti_path}
\end{figure}

This section illustrates the optimal path planning for a pursuer to intercept a moving target, where the pursuer and target are moving at constant speeds in two dimensions. The geometric interpretation for constructing the feasible paths of a pursuer for intercepting a moving target is demonstrated below.  

\noindent Let the pursuer and the target initiate their journey from the points $P(x_{P_0}, y_{P_0}, \theta_{P_0})$ and $T(x_{T_0}, y_{T_0}, \theta_{T_0})$, respectively and after certain period of time the pursuer intersect the target at a arbitrary unknown position. Following Dubin's path the pursuer adheres to a minimum turning radius, allowing it the flexibility to make left or right turns to optimize its path. According to Figure 1, the pursuer can follow either path $PX_1$ or $PX_2$ to travel across a circular trajectory. Subsequently, based on its choice of the left or right turn, the pursuer follows a straight-line path, either $X_1X_3$ or $X_2X_4$, with the aim of intercepting the moving target,  while maintaining a constant speed of $|\bf{V_P}|$. Simultaneously, the target is in motion at a constant speed of $|\bf{V_T}|$, moving towards its path. It is assumed that after a certain time interval $t$, both the pursuer and the target will intersect at a point denoted as $X_3$ or $X_4$. To ensure the success of these intercepts, we make the assumption that the pursuer's speed, $|\bf{V_P}|$, exceeds that of the target, $|\bf{V_T}|$. This creates two possible feasible paths of the type $CS+S_T$ which are as follows:

 \begin{equation}\label{geo_inter}
    CS+S_T\equiv
   \begin{cases}
   LS + S_T\equiv\mbox{arc}\;PX_1 + \mbox{st. line}\;X_1X_3 + \mbox{st. line}\;TX_3,&\text {or},\\
   RS + S_T\equiv\mbox{arc}\;PX_2 + \mbox{st. line}\;X_2X_4 + \mbox{st. line}\;TX_4,
  \end{cases}
\end{equation}

\section{Mathematical Model for Intercepting a Moving Target}\label{modelM}
In this section, we present a mathematical model to determine the most efficient path for a pursuer to intercept a moving target. According to the geometrical interpretation outlined in Section \ref{geometry1}, we formulate mathematical expressions to calculate the concatenated lengths. These lengths comprise two sub-paths $CS$ and $S_T$ types. Moreover, the optimal path for intercepting a moving target falls into one of the two specific cases detailed in the following set discussed in \eqref{geo_inter}.

\begin{equation}
    \{LS+S_T,RS+S_T\}.
\end{equation}

\noindent Let us assume an  initial time $t_0=0$ and a terminal time $t_f=t_3$ (at $X_3$ or $X_4$)  for the both pursuer and target, so that the length of each sub-path can be defined as $\xi_i=|\mathbf{V_P}|(t_i-t_{i-1})$, where $i=1, 2, 3$.  The notation $\xi_i=|\mathbf{V_T}|(t_{i-1}-t_{i})$, where $i=4$ represents the distance traveled by the target until it is intercepted by the pursuer. Now we solve the ordinary differential equations given in \eqref{problem1} for $x_P(t)$, $y_P(t)$, $\theta_P(t)$ during the time interval $t_{i-1}\leq t\leq t_i$, $i=1, 2, 3$ yields, 

\begin{equation}\label{solution1}
\begin{array}{cl} \ds \int_{t_{i-1}}^{t_{i}} \dot{x}_P(t) \, dt & \ =\ds \int_{t_{i-1}}^{t_{i}} |\mathbf{V_P}| \cos\theta_P (t) \,dt,  \\[4mm]
\ds \int_{t_{i-1}}^{t_{i}} \dot{y}_P(t) \, dt & \ =\ds \int_{t_{i-1}}^{t_{i}} |\mathbf{V_P}| \sin\theta_P(t) \,dt,  \\[4mm]
\ds \int_{t_{i-1}}^{t_{i}} \dot{\theta}_P(t) \, dt & \ =\ds \int_{t_{i-1}}^{t_{i}} |\mathbf{V_P}| u (t) \,dt.  \\[4mm]
\end{array}
\end{equation}

\noindent From \eqref{solution1}, we obtain the position $(x_P(t_i),y_P(t_i))$, $i=1,2$ along turning curve ($C$) yields

\begin{equation}\label{equ_6}
\begin{array}{cl} \ds x_P(t_i) & \ =\ds x_P(t_{i-1})+(\sin\theta_P(t_i)-\sin\theta_P(t_{i-1}))/(\dot{\theta}_P(t)/|\mathbf{V_P}|),  \\[4mm]
\ds y_P(t_i) & \ =\ds y_P(t_{i-1})-(\cos\theta_P(t_i)-
   \cos\theta_P(t_{i-1}))/(\dot{\theta}_P(t)/|\mathbf{V_P}|), \\[4mm]
\end{array}
\end{equation}
where 
 \begin{equation}\label{Opt.control}
\hspace{20mm}
\frac{\dot{\theta}_P(t)}{|\mathbf{V_P}|} =
\begin{cases} 
a   &\text{if} \,\, \dot\theta_P(t) > 0 \\
-a  &\text{if}  \,\, \dot\theta_P(t) < 0 \\
0 &\text{if}\,\, \dot\theta_P(t) = 0,
\end{cases} 
\end{equation}

\noindent The position $(x_P(t_i),y_P(t_i)),$ with $i=3$ at which the pursuer intercepts the  moving target along a straight line ($S$) can be described as

\begin{equation}\label{eq_5}
\begin{array}{cl} \ds x_P(t_i) & \ =\ds x_P(t_{i-1})+|\mathbf{V_P}|\cos\theta_P(t_{i-1})(t_i-t_{i-1}), \\[4mm]
\ds y_P(t_i) & \ =\ds y_P(t_{i-1})+|\mathbf{V_P}|\sin\theta_P(t_{i-1})(t_i-t_{i-1}). 
\end{array}
\end{equation}
\noindent Since the speed of the pursuer ($|\mathbf{V_P}|$) is constant over the time duration $[t_{i}-t_{i-1}]$, $~~i=3$, so the length of straight path for pursuer can be written as

\begin{equation}\label{solution3}
  \ds \xi_i=\ds |\mathbf{V_P}| (t_{i}-t_{i-1}), ~~i=3.
\end{equation}

\noindent Now using \eqref{solution3}, we rewrite \eqref{eq_5} as below:
\begin{equation}\label{equ_7}
\begin{array}{cl} \ds x_P(t_i) & \ =\ds x_P(t_{i-1})+\xi_i\cos\theta_P(t_{i-1}), \\[4mm]
\ds y_P(t_i) & \ =\ds y_P(t_{i-1})+\xi_i\sin\theta_P(t_{i-1}), 
\end{array}
\end{equation}

\noindent and heading angles of pursuer along the both curve ($C$) and straight line ($S$) can be obtained as
\begin{equation}\label{theta_solution}
\theta_P(t_i) =
\theta_P(t_{i-1})+(\dot{\theta}_P(t)/|\mathbf{V_P}|)|\mathbf{V_P}|
  (t_i-t_{i-1}), ~~i=1,2,
\end{equation}

\noindent where $u(t)=(\dot{\theta}_P(t)/|\mathbf{V_P}|)=a$, for left-turn circular arc, $u(t)=(\dot{\theta}_P(t)/|\mathbf{V_P}|)=-a$, for right-turn circular arc and $u(t)=(\dot{\theta}_P(t)/|\mathbf{V_P}|)= 0$, for straight line.

\noindent Using \eqref{equ_6}, \eqref{Opt.control}, and \eqref{equ_7}, the optimum path $CS$ of the pursuer for the time duration from $t_0$ to $t_3$ is obtained below.

\begin{equation}\label{path1}
    \ x_P(t_3)-x_P(t_0)=\frac{1}{a}(-\sin\theta_P(t_0)+2\sin\theta_P(t_1)-\sin\theta_P(t_2))+\xi_3\cos\theta_P(t_2),
\end{equation}
and
\begin{equation}\label{path2}
    \ y_P(t_3)-y_P(t_0)=\frac{1}{a}(\cos\theta_P(t_0)-2\cos\theta_P(t_1)+\cos\theta_P(t_2))+\xi_3\sin\theta_P(t_2).
\end{equation} 

\noindent Now, we calculate the target path for the time duration from $t_0$  to $t_3$ as below,

$$\int_{t_0}^{t_3}\dot{x}_T(t) dt=\int_{t_0}^{t_3}|\mathbf{V_T}| \cos \theta_T dt,$$
and 
$$\int_{t_0}^{t_3}\dot{y}_T(t) dt=\int_{t_0}^{t_3}|\mathbf{V_T}| \sin \theta_T dt,$$

\noindent By simplification the above relations and using $\ds \xi_4=\ds |\mathbf{V_T}| (t_3-t_0)$, we obtain the path for moving target as

\begin{equation}\label{path3}
    x_T(t_3)-x_T(t_0) = \xi_4\cos \theta_T,
\end{equation}
and
\begin{equation}\label{path4}
   y_T(t_3)-y_T(t_0) = \xi_4\sin \theta_T. 
\end{equation}

\noindent It is note that 
when a moving target is intercepted by the pursuer, their positions coincide. In other words, $(x_P(t_3),y_P(t_3))=(x_T(t_3),y_T(t_3))$ at the interception moment. As a result combining $\eqref{path1}$-$\eqref{path4}$ and thus we obtain,
\begin{equation}\label{path5}
    \ x_P(t_0)-x_T(t_0)+\frac{1}{a} (-\sin\theta_P(t_0)+2\sin\theta_P(t_1)-\sin\theta_P(t_2))+\xi_3\cos\theta_P(t_2)- \xi_4 \cos\theta_T = 0,
\end{equation}
and
\begin{equation}\label{path6}
   \ y_P(t_0)-y_T(t_0)+\frac{1}{a} (\cos\theta_P(t_0)-2\cos\theta_P(t_1)+\cos\theta_P(t_2))+\xi_3\sin\theta_P(t_2)- \xi_4 \sin\theta_T = 0.
\end{equation}

\noindent Since the pursuer and target move with separate constant velocities and intercept to each other at the same moment, we must include  an additional constraint to the optimization problem as 
\begin{equation}\label{equ_8} |\mathbf{V_T}|\sum_{i=1}^3\xi_i - |\mathbf{V_P}|\xi_4 =0.
\end{equation}

\noindent Considering equations \eqref{theta_solution} and \eqref{path5}-\eqref{equ_8}, we can formulate the following mathematical model for target interception.
\begin{equation}\label{model_1}
 \begin{array}{cl} \min &\ f =\ds \sum_{i=1}^4 \xi_i \\[2mm]
 \mbox{subject to} \\[2mm]
 & \ x_P(t_0)-x_T(t_0)+\ds \frac{1}{a} \left(-\sin\theta_P(t_0)+2\sin\theta_P(t_1)-\sin\theta_P(t_2)\right)+\xi_3\cos\theta_P(t_2)-\xi_4 \cos\theta_T = 0,\\[2mm]
 & \ y_P(t_0)-y_T(t_0)+\ds \frac{1}{a} (\cos\theta_P(t_0)-2\cos\theta_P(t_1)+\cos\theta_P(t_2))+\xi_3\sin\theta_P(t_2)-\xi_4 \sin\theta_T = 0,\\ [2mm]
 & \ |\mathbf{V_T}|(\xi_1+\xi_2+\xi_3) - |\mathbf{V_P}| \xi_4=0,\\ [2mm]
&\ \xi_i \geq 0, \;\;  i=1,2,3,4,
\end{array} 
\end{equation}
where 
\begin{equation*}
    \theta_{P_1}=\theta_{P_0} + a \xi_1, ~~ \theta_{P_2} =\theta_{P_1} - a \xi_2.  
\end{equation*}

\section{Proposed Optimal Path Algorithm}\label{algorithm_target}

 \color{black}
\noindent The possible pathways that described in \eqref{model_1} intersect a moving target satisfy the minimum turn radius criteria established in the following Proposition \eqref{turnradi}.

\begin{figure}[hbt!]
	\begin{minipage}{100mm}
		\begin{center}
			\hspace*{15mm}
			\includegraphics[width=120mm]{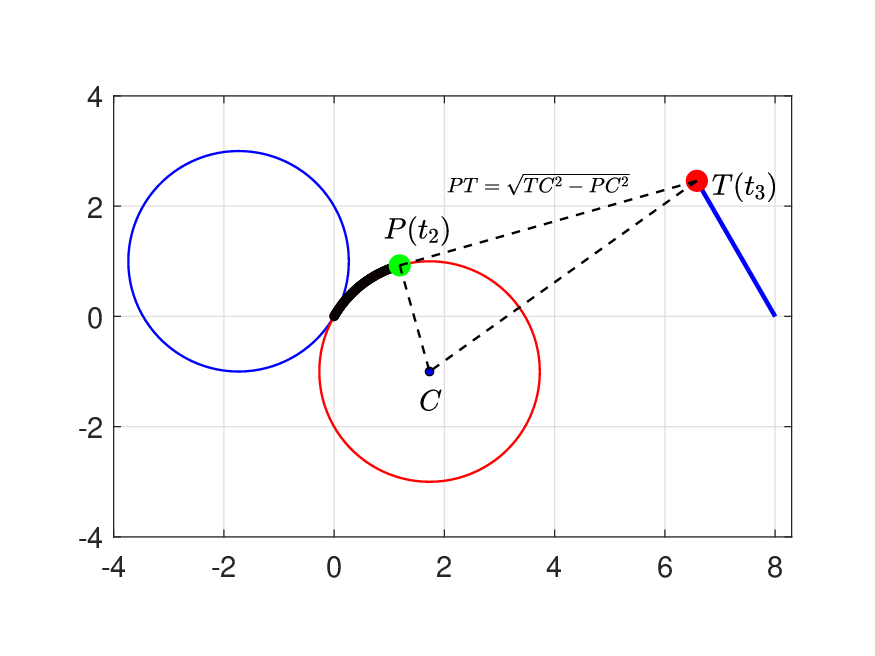} \\[3mm] \vspace{-12mm}
		 \vspace{4mm}
		\end{center}
	\end{minipage}
	\caption{ \small{ Area where pursuer can not intercept target}.}
	\label{fig_tour2}
\end{figure}
\begin{proposition}\label{turnradi}
    A target cannot be intercepted by the pursuer while it is inside its turning circle.
\end{proposition}
\begin{proof}
Let $C$ be the centre of the pursuer's turning circle, and before leaving the turning circle the pursuer's position $P(t)$  is $P(t_2)$. We assume that the pursuer and the target intercept at point $T(t_3)$. The radius of turning circle is $PC$ that perpendicular to the tangent $PT$ and $CT$ is the hypotenuse of the triangle $TPC$. Therefore, from Figure-\ref{fig_tour2} we have $PT=\sqrt{TC^2-PC^2}.$ Since $TC^2-PC^2\geq0$ and $PC \geq 0$, it implies that the target's distance from the center of the turning circle is always greater or equal to the radius of the pursuer's turning circle.
\end{proof}

Now we present the algorithm that implements our proposed model \eqref{model_1} to find the best path for a pursuer intercepting a moving target. In Step 2 of the algorithm, we minimize \eqref{model_1} using specified inputs: curvature $a$, and speeds $\mathbf{|V_P|}$ and $\mathbf{|V_T|}$ for pursuer and target, respectively. The optimal solution determines the lengths of the pursuer's left or right circular paths, along with the straight line from the circular path's exit point to the target intersection point. Using these lengths, we calculate the time required for each of those segments of the paths.

\noindent In Step (3) of the algorithm, we utilize the output obtained in Step (2) to determine the positions of both pursuer and target trajectories at time $t \in [t_o, t_f]$. Initially, in Step (3)(a), we locate the center of the turning radius circle of the pursuer, a necessary step for visualizing the turning curves. Subsequently, in Step (3)(b)(i) and (ii), we compute the positions $(x, y)$ of both the target and pursuer from their starting points up until the intercept occurs.
\newpage 
\begin{algorithm}\label{algorithm}
\begin{description}
\item[Step  1] 
{ \textbf{(Input)}} \\  Choose initial positions  $P(x_{P_0}, y_{P_0}, \theta_{P_0})$ and $T(x_{T_0}, y_{T_0}, \theta_{T_0})$ for a pursuer and a target respectively. Set turning radius $R$, curvature $\ds a=\frac{1}{R}$,  and speeds $\mathbf{|V_P|}$ and $\mathbf{|V_T|}$. Define variables $\xi_i$, $i=1,\dots, 4$.
\item[Step 2] {\textbf{(Determine the length from pursuer's initial point  to intercept point)}}\\
Let left turn angle $\theta_{P_1}=\theta_{P_0} + a \xi_1$ and right turn angle $\theta_{P_2} =\theta_{P_1} - a \xi_2$.  

Find $\xi:=(\bar{\xi}_1, \bar{\xi}_2, \bar{\xi}_3, \bar{\xi}_4)$ that solves Model \eqref{model_1}.

Set $\bar{F}$ := $[ \bar{\xi}_1, \bar{\xi}_2, \bar{\xi}_3, \bar{\xi}_4]$. Set $\bar{T} := [ \bar{t}_1, \bar{t}_2, \bar{t}_3, \bar{t}_4]$, where $\bar{t}_i=\bar{\xi}_i/|\mathbf{V_P}|$, $i=1,2,3.$, $\bar{t}_4=\bar{\xi}_4/|\mathbf{V_T}|$ and $f =\ds \sum_{i=1}^4 \bar{\xi}_i$.
 \item[Step 3] {\textbf{(Simulation)}} 
 \begin{itemize}
     \item [(a)] (Determining turning circle centers)\\
     Set centers for left-turn circle $(\alpha_L, \beta_L)$ and right-turn circle $(\alpha_R, \beta_R)$ where
 \[\alpha_L,~\alpha_R=x_{P_0}\pm R\sin{\theta_{P_0}},\]
 and
  \[\beta_L,~\beta_R=y_{P_0}\pm R\cos{\theta_{P_0}}.\]
  \item [(b)] (Generating points along the path)\\
  Calculate $\theta_{P_1}=\theta_{P_0} + a \bar{\xi}_1$ and $\theta_{P_2} =\theta_{P_1} - a \bar{\xi}_2$.\\
  \begin{itemize}
  \item [(i)] (Points along target trajectory)\\
  \textbf{St line:} Set $\hat{t}_0 = t_0$ and $j=1, \ldots, k.$
   \[\hat{t}_j:=\hat{t}_{j-1} + ( \bar{t}_4- t_0)/k,\]
  \[x_{T_j}=x_T+|\mathbf{V_T}|\cos(\theta_T)( \bar{t}_j- t_0),\]
  and
  \[y_{T_j}=y_T+ |\mathbf{V_T}|\sin(\theta_T)( \bar{t}_j- t_0).\]
  Set $\bar{X}_1=[x_{T_j}, y_{T_j}].$
\item [(ii)] (Points along pursuer trajectory)\\
  \textbf{Left-turn:} Set $\bar{\theta}_{P_0} = \theta_{P_0}$ and $i=1, \ldots, m.$ 
  \[\bar{\theta}_{P_i}:=\bar{\theta}_{P_{i-1}} + (\theta_{P_1}-\theta_{P_0})/m,\]
  \[x_{P_i}=x_{P_0}+ (\sin(\bar{\theta}_{P_i}) - \sin(\theta_{P_0}))/a,\]
  and
  \[y_{P_i}=y_{P_0}- (\cos(\bar{\theta}_{P_i}) - \cos(\theta_{P_0}))/a.\]
  Set $\bar{X}_2=[x_{P_i}, y_{P_i}].$
  
  \textbf{Right-turn:} Set $\hat{\theta}_{P_0} = \theta_{P_0}$ and $i=1, \ldots, m.$ 
  \[\hat{\theta}_{P_i}:=\hat{\theta}_{P_{i-1}} + (\theta_{P_2}-\theta_{P_0})/m,\]
  \[x_{P_i}=x_{P_0}- (\sin(\hat{\theta}_{P_i}) - \sin(\theta_{P_0}))/a,\]
  and
  \[y_{P_i}=y_{P_o}+ (\cos(\hat{\theta}_{P_i}) - \cos(\theta_{P_0}))/a.\]
  Set $\bar{X}_3=[x_{P_i}, y_{P_i}].$

  \textbf{St. Line:} Use \eqref{equ_7} and similar steps stated in Step~(b)(i).\\
  Set $\bar{X}_4=[x_{P_i}, y_{P_i}].$
  \end{itemize}
  
 \end{itemize}
 \item[Step 4] {\textbf{(Output)}} \\
  Target trajectory $\bar{X}_1$ and Pursuer trajectory $\bar{X}_2 \cup \bar{X}_3 \cup \bar{X}_4$.
\end{description}
    
\end{algorithm}
\color{black}
\section{ Numerical Experiments, Results and Analysis}\label{sec6}
In this section, we present the results obtained from the numerical experiments to validate the efficiency of our proposed model \eqref{model_1} for determining the optimal path for a pursuer to intercept a moving target. We provide initial locations and heading angles for both the pursuer and the target, considering their constant velocities, in order to calculate the optimal path for target interception.

\noindent In Example 1, we tested our proposed Algorithm \eqref{algorithm} with various combinations of pursuer and target positions, velocities, and turning radii. The pursuer's objective is to intercept the target in the shortest possible distance or time. As the turning radius is influenced by both angular and linear speeds, Example 2 tested the model, calculating the turning radius using the parameters such as pursuer and target positions, and  velocities. We explored different optimal path combinations by adjusting angular speeds, both increasing and decreasing them, to evaluate their impact on interception strategies. In experiments conducted in Examples 1 and 2, Algorithm \eqref{algorithm} demonstrated remarkable efficiency in computing optimal paths, as evidenced by Figures~\ref{PurTar1}--\ref{PurTar5} and the detailed results presented in Tables \ref{table:exp-1}-\ref{table:exp-5}.

\noindent In the experiments, we have implemented Algorithm \eqref{algorithm} writing codes in AMPL \cite{Fourer2003}, utilizing various solvers such as Ipopt \cite{Wachter2006} and KNITRO \cite{Knitro2006}, all with default settings, to solve the problems. Almost both solvers exhibited a high convergence rate in solving the proposed model. The obtained solutions are simulated by using MATLAB for the numerical experiments. In our analysis, the computations are conducted on an HP Evo laptop with 16 GB RAM and a Core i7 processor running at 4.6 GHz.

\noindent \textbf{Example 1} \label{example1}

\noindent Let an initial position for the pursuer, denoted as ~$P(0,0, 2\pi/3)$ and a target point $T(-5,0,\pi/2)$. The objective is for the pursuer $P$, moving at a constant speed $\mathbf{|V_P|}$, to intercept the target $T$, which is also moving at a constant speed $\mathbf{|V_T|}$. This interception is accomplished by finding the shortest path $f$, as determined by Algorithm \eqref{algorithm}.  It is important to note that we assume a minimum turning radius of $R=1$ for the pursuer's turning curve.  The pursuer takes time $t$ to intercept the target. 

\noindent The results of our analysis reveal that the $LS+S_T$ path proves to be the optimal choice to successfully achieve the task, as illustrated in Figure \ref{PurTar1}(a). Another feasible path is the $RS+S_T$ path, showcased in Figure \ref{PurTar1}(b). Table \ref{table:exp-1} provides a detailed breakdown of the lengths and time periods associated with each segment of these paths. Specifically, the total length of the path obtained by our algorithm $f=6.26$, and the corresponding time is $t=1.04$. In particular, the non-optimal path requires two times more length ($f=14.16$) and time ($t=2.36$), when compared to the optimal path, as illustrated in Table \ref{table:exp-1} and demonstrated in Figure \ref{PurTar1}(a) and (b).
The computational time taken by our system to derive the optimum path is around $0.078$ seconds.


  \par
	\begin{table}[!htbp]
		\caption{\small{\textit{Model Performance Variations for Optimal Path Generations. }}}
		\footnotesize
		\vskip 1.5em
		\centering
		\begin{tabular}{|c| c| c| c| c| c| c| c|c|c|c| }
  \hline
\multicolumn{1}{|c|}{$\mathbf{|V|}$ $-$ $R$}& \multicolumn{4}{|c|}{Length} & \multicolumn{4}{|c|}{Time} & \multicolumn{2}{|c|}{} \\
\hline
			\hline
			$\mathbf{|V_P|}=5$ & Left  &  Right  & St.		    	& Target & Left& Right& St. & Target & Total & Total \\
			$\mathbf{|V_T|}=1$ &     turn		&   turn  	& line & St. line & turn& turn &line& St. line & length & time  \\
			$\;~R\;\;=1$&$\bar{\xi}_1$ &     $\bar{\xi}_2$ 		&  $\bar{\xi}_3$   	& $\bar{\xi}_4$ & $\bar{t}_1$& $\bar{t}_2$ & $\bar{t}_3$ &$\bar{t}_4$ & $f$ & $t$ \\ [0.5ex]
			\hline
			Optimal Path & 0.92  & 0 & 4.30 & 1.04    &	0.18  & 0& 0.86 	& 1.04 &6.26 & 1.04 \\ [.5ex]\hline
			Feasible Path & 0 & 5.70   & 6.09 &2.36 & 0 &	1.14	& 1.21   &2.36 &14.16 & 2.36\\ [.5ex]\hline				
		\end{tabular}
		\label{table:exp-1}
	\end{table}
\begin{figure}[!htbp]
\vspace*{-20mm}
 \centering 
 
\includegraphics[width=130mm]{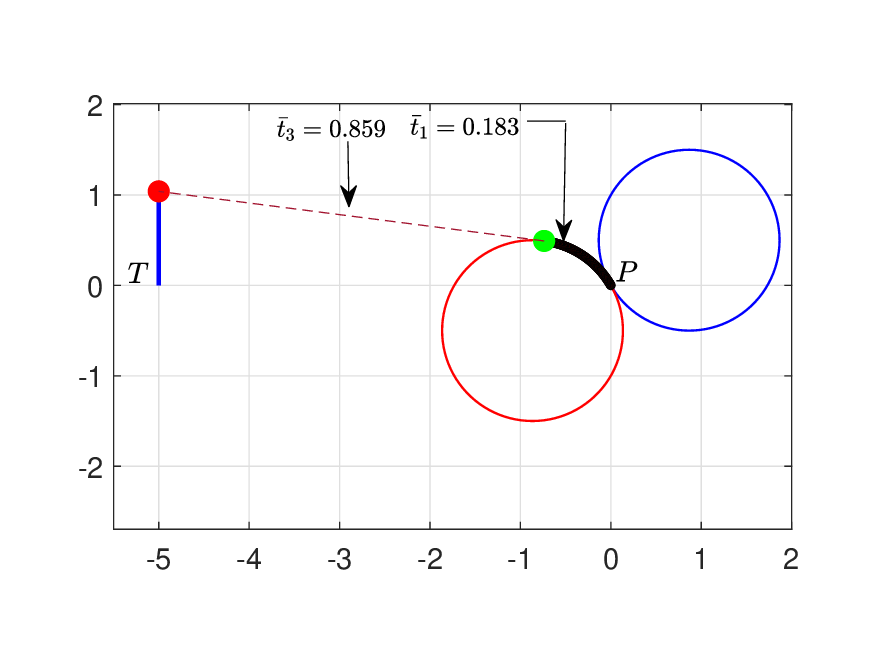} \\ \vspace{-.6cm}
{\footnotesize (a) Optimal path}.
\\
 \centering
 \hspace{.4cm}
\includegraphics[width=130mm]{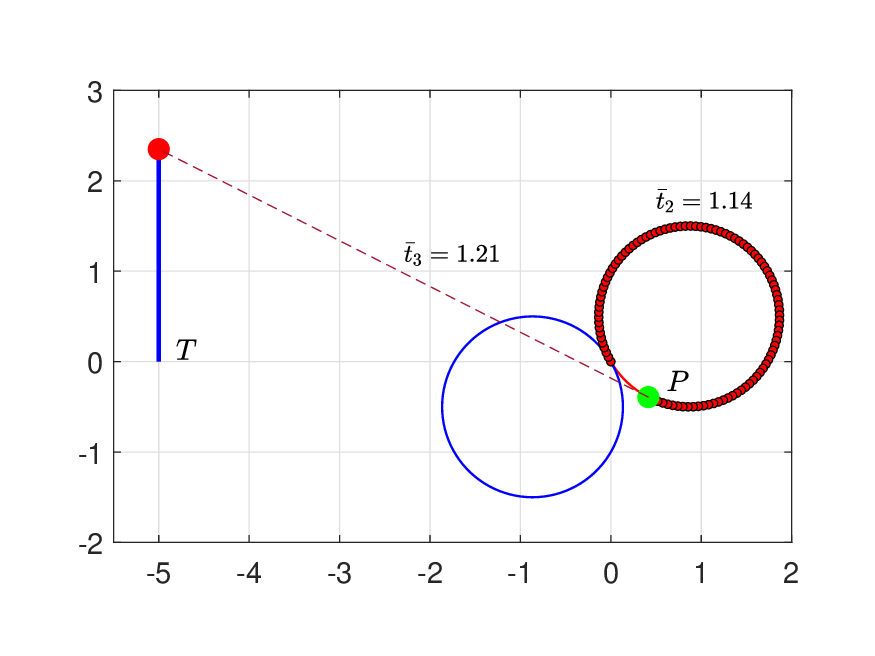} \\
{\footnotesize (b) Feasible path}.
\\[5mm]
\caption{\sf\small Optimizing the path of the pursuer, starting at position $P(0,0)$ with a heading angle of $\theta_P = 2\pi/3$, to intercept target $T(-5,0)$, which is located to the left of the pursuer and heading in the direction of $\theta_T = \pi/2$.}
\label{PurTar1}
\end{figure}



\noindent Now, we fixed the target point to the right side of the pursuer, specifically at coordinate $T(8, -2, 2\pi/3)$, while increasing the target's speed to $\mathbf{|V_T|}=2$. Additionally, we adjusted the initial angle of the pursuer to $\pi/3$. The pursuer's turning curve meets a minimum radius of $R=1$.

\noindent Our analysis concluded that the optimal choice to successfully achieve this task is the $RS+S_T$ path, as shown in Figure \ref{PurTar2}(a). Other feasible path is $LS+S_T$, displayed in Figure \ref{PurTar2}(b). Table \ref{table:exp-2} provides a comprehensive breakdown of the lengths and time intervals associated with each segment of these paths. Specifically, our algorithm yielded an optimal path length of $f=9.55$ and a corresponding time of $t=1.36$. Almost the previous analysis, the non-optimal path requires twice the length and time compared to the optimal path, as depicted in Figure \ref{PurTar2}(a) and \ref{PurTar2}(b). The computational time needed for our system to determine the optimal path is around $0.093$ seconds.


  \par
	\begin{table}[!htbp]
		\caption{\small{\textit{Model Performance Variations for Optimal Path Generations. }}}
		\footnotesize
		\vskip 1.5em
		\centering
		\begin{tabular}{|c| c| c| c| c| c| c| c|c|c|c| }
  \hline
\multicolumn{1}{|c|}{$\mathbf{|V|}$ $-$ $R$}& \multicolumn{4}{|c|}{Length} & \multicolumn{4}{|c|}{Time} & \multicolumn{2}{|c|}{} \\
\hline
			\hline
			$\mathbf{|V_P|}=5$ & Left  &  Right  & St.		    	& Target & Left& Right& St. & Target & Total & Total \\
			$\mathbf{|V_T|}=2$ &     turn		&   turn  	& line & St. line & turn& turn &line& St. line & length & time  \\
			$\;~R\;\;=1$&$\bar{\xi}_1$ &     $\bar{\xi}_2$ 		&  $\bar{\xi}_3$   	& $\bar{\xi}_4$ & $\bar{t}_1$& $\bar{t}_2$ & $\bar{t}_3$ &$\bar{t}_4$ & $f$ & $t$ \\ [0.5ex]
			\hline
			Optimal Path & 0  & 1.07 & 5.75 & 2.73    &	0  & 0.21& 1.15 	& 1.36 &9.55 & 1.36 \\ [.5ex]\hline
			Feasible Path & 5.65 & 0   & 6.63 &4.91 & 1.13 &	0	& 1.33   &2.46 &17.19 & 2.46\\ [.5ex]\hline

		\end{tabular}
		\label{table:exp-2}
	\end{table}
\begin{figure}[hbt!]
\vspace*{-15mm}
 \centering
\includegraphics[width=120mm]{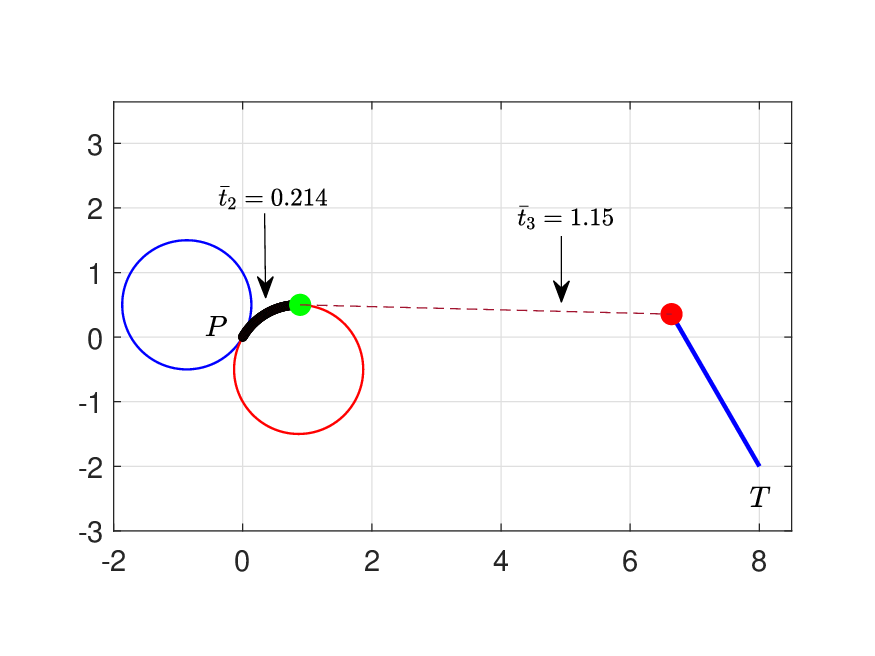} \\ \vspace{-.8cm}
{\footnotesize (a) Optimal path}.
\\
 \centering
 \hspace{-.4cm}
\includegraphics[width=120mm]{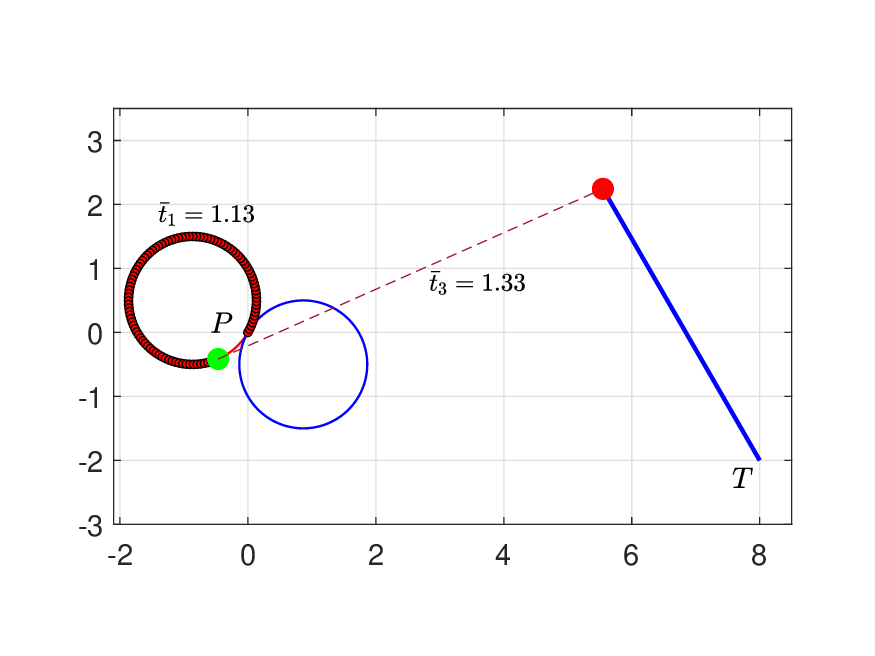} \\ \vspace{-1 cm}
{\footnotesize (b) Feasible path}.
\\[5mm]
\caption{\sf\small Optimizing the path of the pursuer, starting at position $P(0,0)$ with a heading angle of $\theta_P = \pi/3$, to intercept target $T(8,-2)$, which is located to the right of the pursuer and heading in the direction of $\theta_T = 2\pi/3$.}
\label{PurTar2}
\end{figure}

\noindent We further evaluated the adaptability of Algorithm \eqref{algorithm} by conducting tests in various directions, including scenarios where the target travels horizontally relative to the pursuer. The analysis for these scenarios is presented in Figure \ref{PurTar3}(a)(b) and data summarized in Table \ref{table:exp-3}, considering an initial position at $P(0,0, \pi/3)$ and $T(8,3, \pi)$. 
  \par
	\begin{table}[!htbp]
		\caption{\small{\textit{Model Performance Variations for Optimal Path Generations. }}}
		\footnotesize
		\vskip 1.5em
		\centering
		\begin{tabular}{|c| c| c| c| c| c| c| c|c|c|c| }
  \hline
\multicolumn{1}{|c|}{$\mathbf{|V|}$ $-$  $R$}& \multicolumn{4}{|c|}{Length} & \multicolumn{4}{|c|}{Time} & \multicolumn{2}{|c|}{} \\
\hline
			\hline
			$\mathbf{|V_P|}=5$ & Left  &  Right  & St.		    	& Target & Left& Right& St. & Target & Total & Total \\
			$\mathbf{|V_T|}=2$ &     turn		&   turn  	& line & St. line & turn& turn &line& St. line & length & time  \\
			$\;~R\;\;=1$&$\bar{\xi}_1$ &     $\bar{\xi}_2$ 		&  $\bar{\xi}_3$   	& $\bar{\xi}_4$ & $\bar{t}_1$& $\bar{t}_2$ & $\bar{t}_3$ &$\bar{t}_4$ & $f$ & $t$ \\ [0.5ex]
			\hline
			Optimal Path & 0  & 0.57 & 5.71 & 2.51    &	0  & 0.11& 1.14 	& 1.26 &8.80 & 1.26 \\ [.5ex]\hline
			Feasible Path & 5.94 & 0 & 5.03   & 4.39 &1.19 & 0 &	1.00	& 2.20   &15.36 &2.20 \\ [.5ex]\hline

		\end{tabular}
		\label{table:exp-3}
	\end{table}
\begin{figure}[!htbp]
\vspace*{-20mm}
 \centering
\includegraphics[width=130mm]{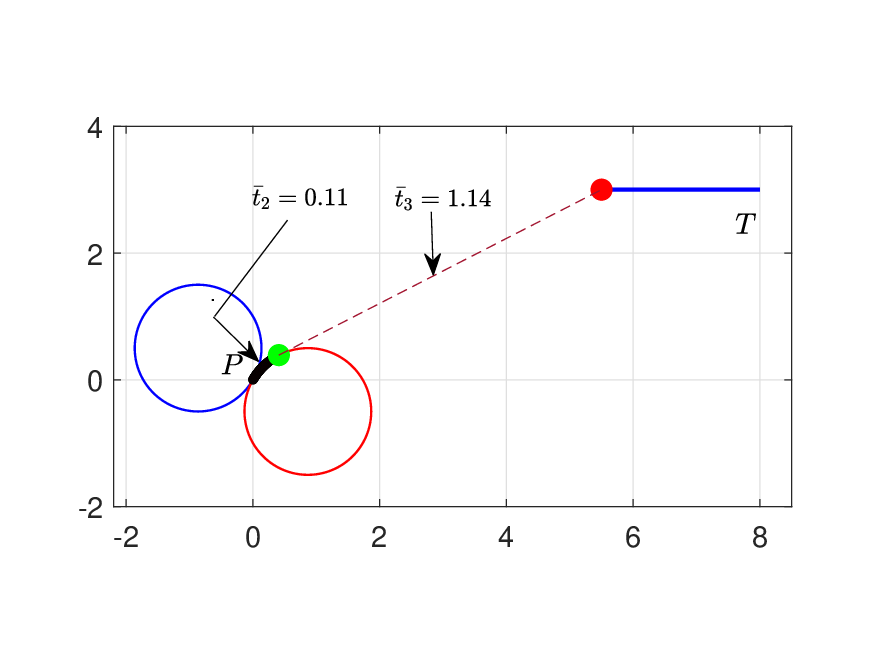} \\ \vspace{-.9cm}
{\footnotesize (a) Optimal path}.
\\ 
 \centering
\includegraphics[width=130mm]{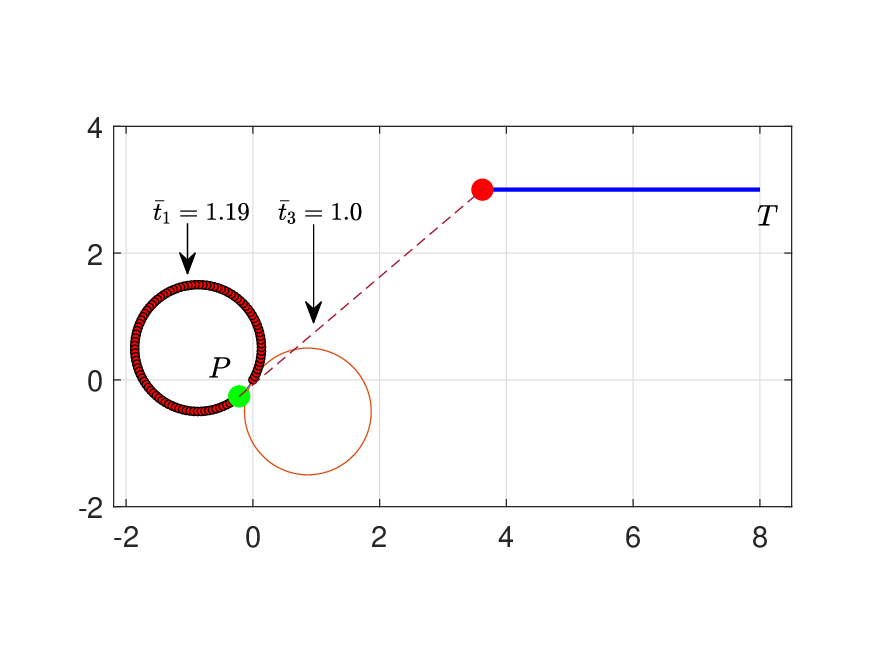} \\ \vspace{-.9cm}
{\footnotesize (b) Feasible path}.
\\[5mm]
\caption{\sf\small Optimizing the path of the pursuer, starting at position $P(0,0)$ with a heading angle of $\theta_P = \pi/3$, to intercept target $T(8,3)$, which is located to the right of the pursuer and heading in the direction of $\theta_T = \pi$.}
\label{PurTar3}
\end{figure}
\noindent In addition, we tested our proposed model \eqref{model_1} for larger turning radius $R=3$ and  various target directions which demonstrated in Figure \ref{PurTar4}(a)(b) and detailed of the results depicted in Table \ref{table:exp-4}, maintaining the pursuer and target positions at  $P(0,0, \pi/3)$ and $T(8,3, \pi)$, and also for the positions $P(0,0, \pi/3)$ and $T(-1,10, 3\pi/2)$ .
  \par
	\begin{table}[!htbp]
		\caption{\small{\textit{Model Performance Variations for Optimal Path Generations. }}}
		\footnotesize
		\vskip 1.5em
		\centering
		\begin{tabular}{|c| c| c| c| c| c| c| c|c|c|c| }
  \hline
\multicolumn{1}{|c|}{$\mathbf{|V|}$ $-$ $R$}& \multicolumn{4}{|c|}{Length} & \multicolumn{4}{|c|}{Time} & \multicolumn{2}{|c|}{Optimal} \\
\hline
			\hline
			$\mathbf{|V_P|}=5$ & Left  &  Right  & St.		    	& Target & Left& Right& St. & Target & Total & Total \\
			$\mathbf{|V_T|}=2$ &     turn		&   turn  	& line & St. line & turn& turn &line& St. line & length & time  \\
			$\;~R\;\;=3$&$\bar{\xi}_1$ &     $\bar{\xi}_2$ 		&  $\bar{\xi}_3$   	& $\bar{\xi}_4$ & $\bar{t}_1$& $\bar{t}_2$ & $\bar{t}_3$ &$\bar{t}_4$ & $f$ & $t$ \\ [0.5ex]
			\hline
			$\tiny P(0,0)~~ \to ~~T(8,3)$ & 0  & 1.92 & 4.41 & 2.53    &	0  & 0.38& 0.88 	& 1.27 &8.87 & 1.27 \\ [.5ex]\hline
			$P(0,0) \to T(-1,10)$ & 2.37 & 0 & 4.96   & 2.93 &0.47 & 0 &	0.99	& 1.46   &10.25 &1.46 \\ [.5ex]\hline

		\end{tabular}
		\label{table:exp-4}
	\end{table}

\begin{figure}[!htbp]
\vspace*{-20mm}

 \centering
\includegraphics[width=110mm]{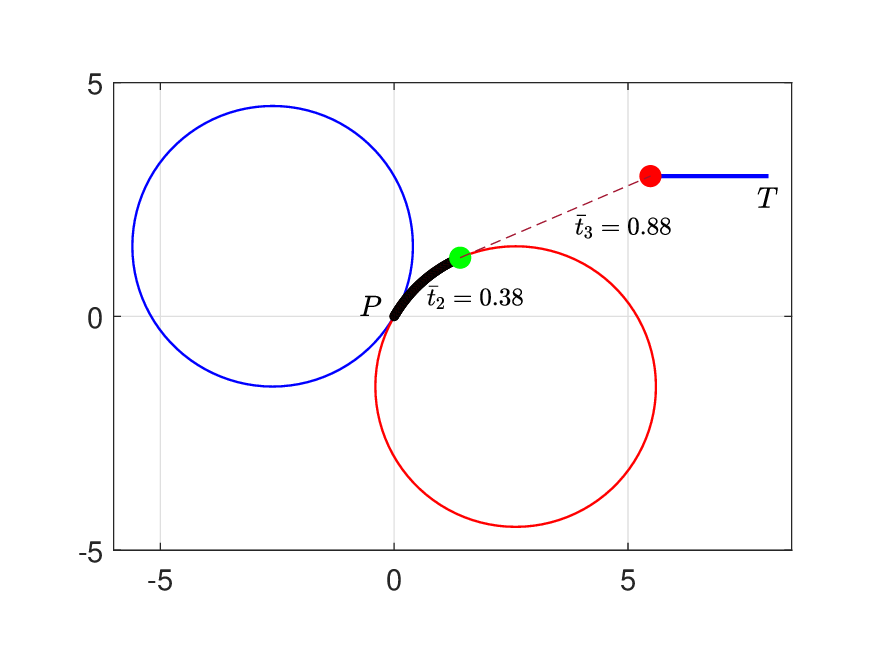} \\ \vspace{-.5cm}
{\footnotesize (a) Optimal path for pursuer $P(0,0, \pi/3)$ and target $T(8,3, \pi)$}.
\\[1mm]
 \centering
\includegraphics[width=130mm]{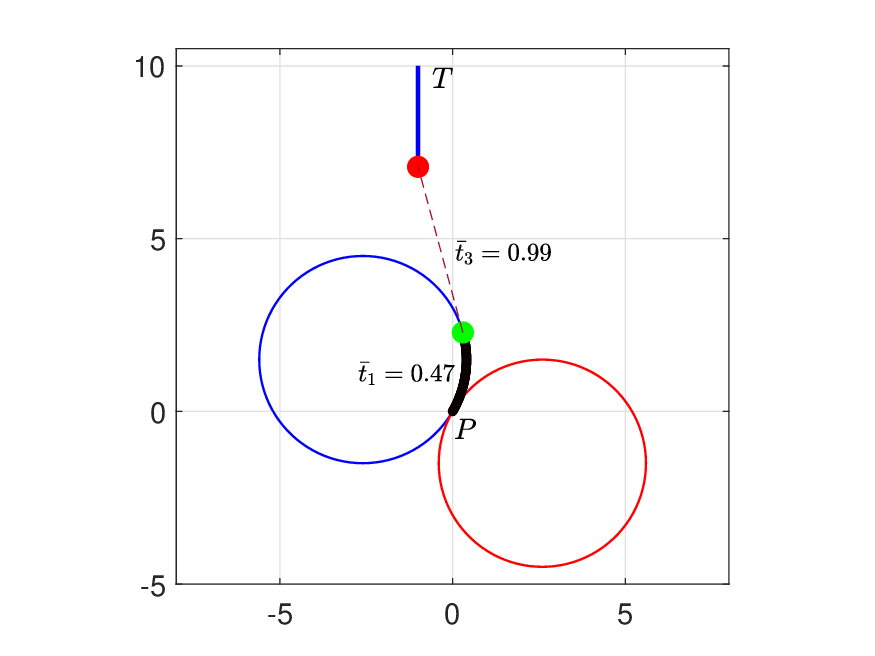} \\
{\footnotesize (b) Optimal path for pursuer $P(0,0, \pi/3)$ and target $T(-1,10, 3\pi/2)$}.
\\[5mm]
\caption{\sf\small Optimizing pursuer's path: exploring different target directions with radius $R=3$.}
\label{PurTar4}
\end{figure}
\newpage
\begin{remark}
    In our analysis, we aim to optimize the total physical distance traveled by both the pursuer and the target. However, we can also consider optimizing the time, focusing on reaching the interception of the target as quickly as possible. It's important to note that these two objectives don't always align, especially when acceleration is involved along the path. In our current model, we assume that both the pursuer and the target move with constant velocity. In this context, optimizing the length or optimizing the time both give us the same outcome.

\end{remark}

\subsection{Model Validation: Turning Radius via Speed Dynamics}

The turning radius of the pursuer depends on both its angular speed (the rate at which it turns or rotates) and its linear speed (its speed in a tangent line).
Specifically, the turning radius can be calculated using the relation $\ds R=\frac{|\mathbf{V}|}{|\dot{\theta}(t)|}$ where $|\mathbf{V}|$ is linear speed and $|\dot{\theta}(t)|$ is angular speed.
A higher angular velocity (faster turn) or a lower linear velocity (slower straight-line speed) will result in a smaller turning radius, meaning the object or vehicle can make a tighter turn. Conversely, a lower angular velocity or a higher linear velocity will result in a larger turning radius, requiring a wider turn.  Keeping this mind we do further analysis of Algorithm \eqref{algorithm} through the Example-2.

\noindent\textbf{Example 2}\label{example2} 

\noindent Let an initial position for the pursuer, denoted as ~$P(0,0, 2\pi/3)$ and a target point ~$T(-100,0,\pi/2)$. The objective is obtaining a minimum path $f$ chosen by \eqref{model_1}. Now we consider a combination of angular speeds: $|\dot{\theta}(t)|<1$, $|\dot{\theta}(t)|=1$, and $|\dot{\theta}(t)|>1$, with a fixed linear pursuer speed $|\mathbf{V_P}|=12$.

\noindent Based on the parameters mentioned above, the turning radius $R$ is calculated using the relation $\ds R=\frac{|\mathbf{V}|}{|\dot{\theta}(t)|}$, and optimal paths are obtained by model \eqref{model_1} which depicted in Figure \ref{PurTar5}. We observe that when $|\dot{\theta}(t)|<1$, the turning radius is $R=48$, and the optimal path length $f=159.65$  is achieved, covering one-third of the total length with the turning curve. Conversely, when $|\dot{\theta}(t)| >1$, the turning radius is $R=3$, and the optimal path length $f=156.04$ occurs when the pursuer turns left, but only $1/100$ of the total length is covered by the turning curve. These obtained results are tabulated in Table \ref{table:exp-5}, and the optimal paths for these different combinations can be seen in Figure \ref{PurTar5}. Overall, it appears that improved results are obtained when angular velocity is increased, resulting in a smaller turning radius.
  \par
	\begin{table}[!htbp]
		\caption{\small{\textit{Model Performance Variations with Changing Angular Speed  $|\dot{\theta}(t)|$. }}}
		\footnotesize
		\vskip 1.5em
		\centering
		\begin{tabular}{|c| c| c| c| c| c| c| c|c|c|c| }
  \hline
\multicolumn{1}{|c|}{$\mathbf{|V_P|}$, $|\dot{\theta}(t)|$ and $R$}& \multicolumn{4}{|c|}{Length} & \multicolumn{4}{|c|}{Time} & \multicolumn{2}{|c|}{Optimal} \\
\hline
			\hline
			 & Left  &  Right  & St.		    	& Target & Left& Right& St. & Target & Total & Total \\
			$\mathbf{|V_P|}=12$ &     turn		&   turn  	& line & St. line & turn& turn &line& St. line & length & time  \\
		$R=\mathbf{|V_P|}/|\dot{\theta}(t)|$	&$\bar{\xi}_1$ &     $\bar{\xi}_2$ 		&  $\bar{\xi}_3$   	& $\bar{\xi}_4$ & $\bar{t}_1$& $\bar{t}_2$ & $\bar{t}_3$ &$\bar{t}_4$ & $f$ & $t$ \\ [0.5ex]
			\hline
			$|\dot{\theta}(t)|=1/4$ $|$ $R=48$ & 34.31  & 0 & 78.39 & 46.96    &	2.86  & 0& 6.53 	& 9.39 &159.65 & 9.39 \\ [.5ex]\hline
			$|\dot{\theta}(t)|=1/1$ $|$ $R=12$ & 7.64 & 0   & 102.94 &46.08 & 0.64 &	0	& 8.58   &9.22 &156.66 & 9.22\\ [.5ex]\hline
		$|\dot{\theta}(t)|=4/1$ $|$ $R=3$ & 1.87 & 0   & 108.28 &45.89 & 0.16 &	0	& 9.02   &9.18 & 156.04&9.18\\ [.5ex]\hline
				
		\end{tabular}
		\label{table:exp-5}
	\end{table}

\begin{figure}[!htbp]
\hspace{-1cm}
\begin{minipage}{90mm}
\begin{center}
\hspace*{0cm}
\includegraphics[width=90mm]{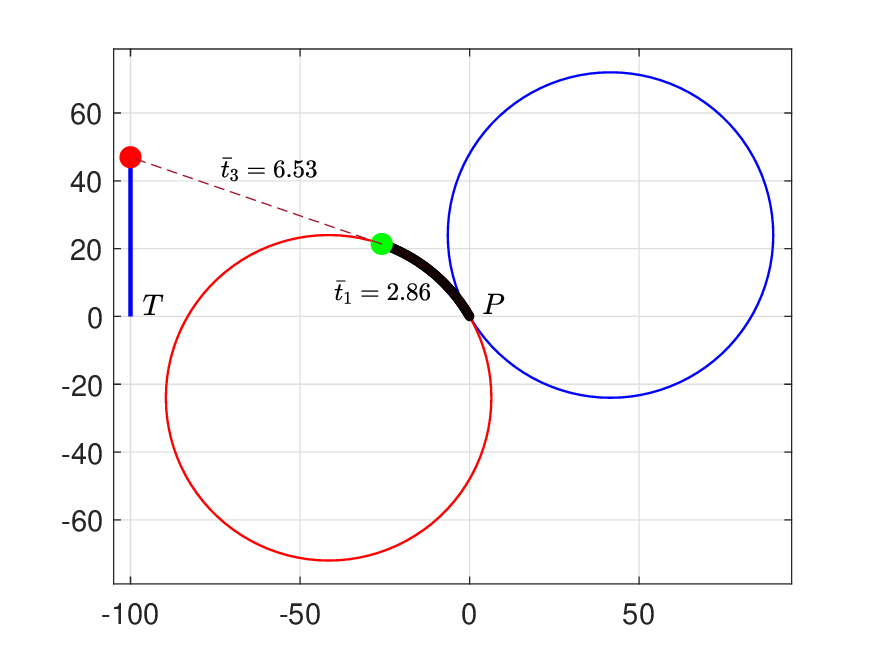} \\ \hspace{-3mm}
{\footnotesize (a) Optimal path when $|\dot{\theta}(t)| <1$ and $\mathbf{|V_P|}=12$.}\hspace{-3mm}
\end{center}
\end{minipage}
\hspace{-2mm}
\begin{minipage}{90mm}
\begin{center}
\hspace*{0cm}
\includegraphics[width=100mm]{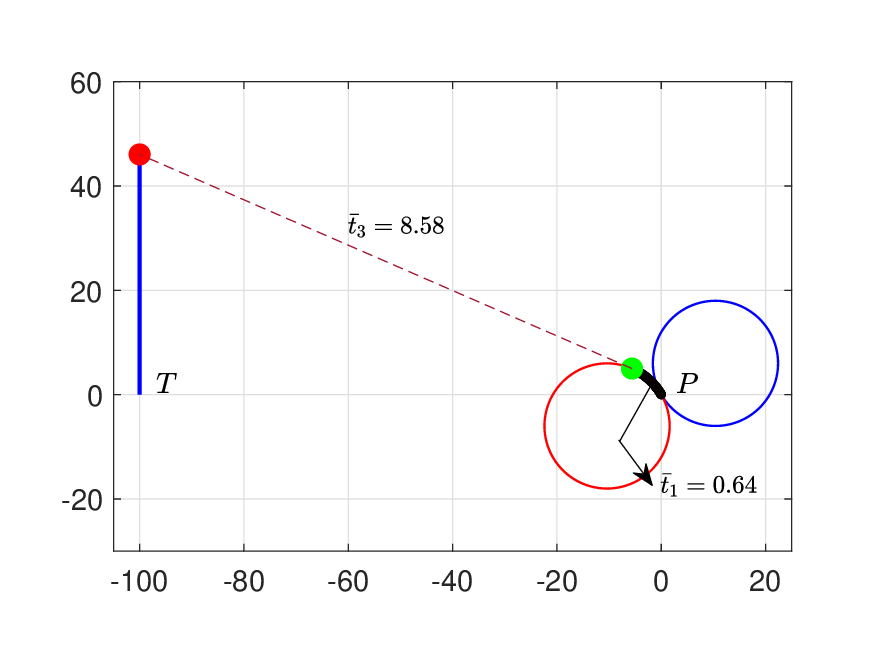} \\ \hspace{-2mm}
{\footnotesize (b) Optimal path when $|\dot{\theta}(t)| =1$ and $\mathbf{|V_P|}=12$.}\hspace{-2mm}
\end{center}
\end{minipage}
\\\hspace{-2mm}
\hspace*{-1cm}
\begin{minipage}{90mm}
\begin{center}
\hspace*{0cm}
\includegraphics[width=90mm]{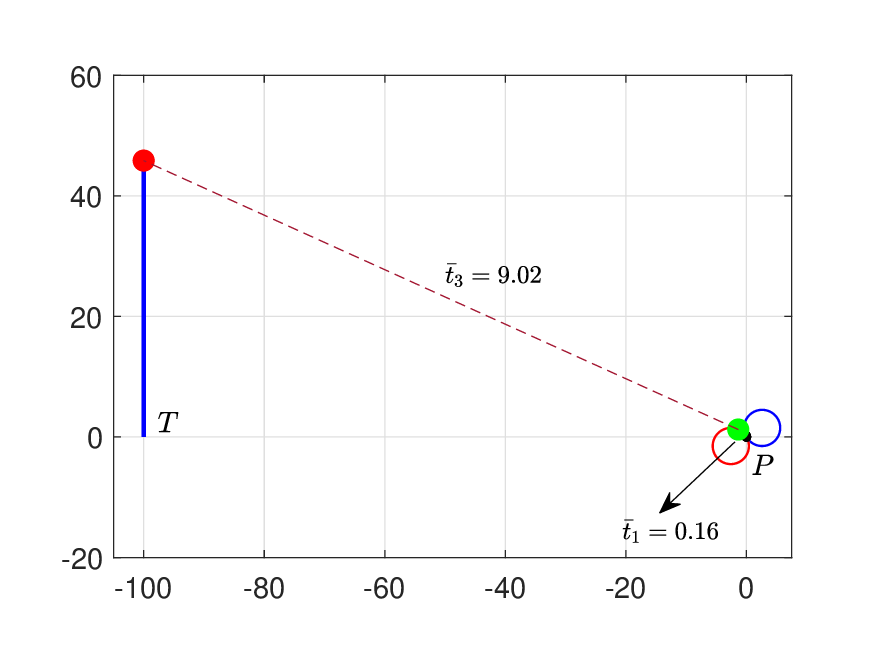} \\ \hspace{-2mm}
{\footnotesize (c) Optimal path when $|\dot{\theta}(t)| >1$ and $\mathbf{|V_P|}=12$.}
\end{center}
\end{minipage}
\vspace{.5cm}
\caption{\sf\small Optimizing the path of the pursuer, starting at position $P(0,0)$ with a heading angle of $\theta_P = 2\pi/3$, to intercept target $T(-100,0)$, which is located to the left of the pursuer and heading in the direction of $\theta_T = \pi/2$. }
\label{PurTar5}
\end{figure}

\noindent We also validated the proposed model in some special scenarios. For instance, when the target's starting position lies within the pursuer's turning radius. In this situation, according to Proposition \ref{turnradi}, the pursuer cannot intercept the target within the turning radius. However, interception occurs outside the circle when an appropriate velocity combination is used. The results for this case are presented in the 3rd row of Table \ref{table:exp-6} and demonstrated in Figure \ref{comparison}(a).

\noindent It is assumed that the pursuer's speed is strictly greater than the target's speed. However, this depends on the heading directions for the target trajectory and other parameters defined in the algorithm. To test this, an example is given for the proposed model, where the pursuer and target have the same velocities. The results for this case are provided in the 4th row of Table \ref{table:exp-6} and depicted in Figure \ref{comparison}(b).

\par
	\begin{table}[!htbp]
		\caption{\small{\textit{Special Cases. }}}
		\footnotesize
		\vskip 1.5em
		\centering
		\begin{tabular}{|c| c| c| c| c| c| c| c|c|c|c| }
  \hline
\multicolumn{1}{|c|}{$\mathbf{|V_P|}$, $\mathbf{|V_T|}$, }& \multicolumn{4}{|c|}{Length} & \multicolumn{4}{|c|}{Time} & \multicolumn{2}{|c|}{Optimal} \\
\multicolumn{1}{|c|}{ $|\dot{\theta}(t)|$ and $R$}& \multicolumn{4}{|c|}{} & \multicolumn{4}{|c|}{} & \multicolumn{2}{|c|}{} \\
\hline
			\hline
			 & Left  &  Right  & St.		    	& Target & Left& Right& St. & Target & Total & Total \\
			$R=\mathbf{|V_P|}/|\dot{\theta}(t)|$ &     turn		&   turn  	& line & St. line & turn& turn &line& St. line & length & time  \\
			&$\bar{\xi}_1$ &     $\bar{\xi}_2$ 		&  $\bar{\xi}_3$   	& $\bar{\xi}_4$ & $\bar{t}_1$& $\bar{t}_2$ & $\bar{t}_3$ &$\bar{t}_4$ & $f$ & $t$ \\ [0.5ex]
			\hline
             $\mathbf{|V_P|}=12$, $\mathbf{|V_T|}=4$&&&&&&&&&&\\
			$|\dot{\theta}(t)|=1/4$ $|$ $R=48$ & 46.71  & 0 & 32.04 & 26.26    &	3.89  & 0& 2.67 	& 6.56 &105.01 & 6.56 \\ [.5ex]\hline
			 $\mathbf{|V_P|}=8$, $\mathbf{|V_T|}=8$&&&&&&&&&&\\$|\dot{\theta}(t)|=1/6$ $|$ $R=48$ & 27.47 & 0   & 117.68 &145.18 & 3.43 &	0	& 14.71   &18.14 &290.33 & 18.14\\ [.5ex]\hline
				
		\end{tabular}
		\label{table:exp-6}
	\end{table}

\newpage

\subsection{Comparison with the Model Presented in \cite{Looker2008}.}

We compare the model \eqref{model_1} with the model presented in \cite{Looker2008} with identical parameter settings. The initial positions of the pursuer and target are \((1, 0)\) and \((2, 2)\),  with velocities of \(10\) and \(1\), and heading angles of $0$ and \(5\pi/4\), respectively. Since \cite{Looker2008} only reports the total time taken for target interception, we limit our comparison to this metric. Our proposed model achieves interception in \(0.24\) units of time, where as the model in \cite{Looker2008} requires \(2.37\) units. This represents a significant improvement over the method introduced in \cite{Looker2008}, clearly demonstrating its effectiveness in reducing time. Figure \ref{comparison}(c) illustrates the optimal path for the proposed model, which can be compared with Figure 4(b) in \cite{Looker2008}. Additionally, the proposed model's total path length for interception is  \(2.66\) units. Detailed results for the proposed model, including the length and time for each path segment, are presented in the 3rd row of Table \ref{table:exp-7}. 

\par
	\begin{table}[!htbp]
		\caption{\small{\textit{Model Performance Variations for Comparison. }}}
		\footnotesize
		\vskip 1.5em
		\centering
		\begin{tabular}{|c| c| c| c| c| c| c| c|c|c|c| }
  \hline
\multicolumn{1}{|c|}{$\mathbf{|V_P|}$, $\mathbf{|V_T|}$, }& \multicolumn{4}{|c|}{Length} & \multicolumn{4}{|c|}{Time} & \multicolumn{2}{|c|}{Optimal} \\
\multicolumn{1}{|c|}{ $|\dot{\theta}(t)|$ and $R$}& \multicolumn{4}{|c|}{} & \multicolumn{4}{|c|}{} & \multicolumn{2}{|c|}{} \\
\hline
			\hline
			 & Left  &  Right  & St.		    	& Target & Left& Right& St. & Target & Total & Total \\
			$R=\mathbf{|V_P|}/|\dot{\theta}(t)|$ &     turn		&   turn  	& line & St. line & turn& turn &line& St. line & length & time  \\
			&$\bar{\xi}_1$ &     $\bar{\xi}_2$ 		&  $\bar{\xi}_3$   	& $\bar{\xi}_4$ & $\bar{t}_1$& $\bar{t}_2$ & $\bar{t}_3$ &$\bar{t}_4$ & $f$ & $t$ \\ [0.5ex]
			\hline
		 $\mathbf{|V_P|}=10$, $\mathbf{|V_T|}=1$&&&&&&&&&&\\$|\dot{\theta}(t)|=10/1$ $|$ $R=1$ & 1.81 & 0   & 0.61 &0.24 & 0.18 &	0	& 0.06   &0.24 & 2.66&0.24\\ [.5ex]\hline
				
		\end{tabular}
		\label{table:exp-7}
	\end{table}

\begin{figure}[!htbp]
\hspace{-1cm}
\begin{minipage}{95mm}
\begin{center}
\hspace*{0cm}
\includegraphics[width=95mm]{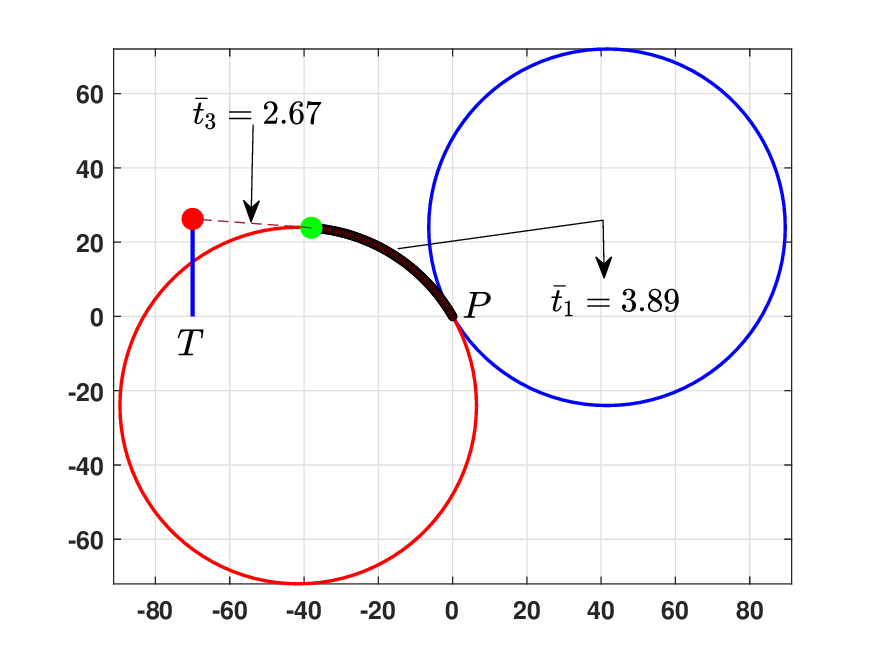} \\ \hspace{-3mm}
{\footnotesize (a) Optimal path when target inside \\ the turning circle.}\hspace{-3mm}
\end{center}
\end{minipage}
\hspace{-2mm}
\begin{minipage}{90mm}
\begin{center}
\hspace*{0cm}
\includegraphics[width=90mm]{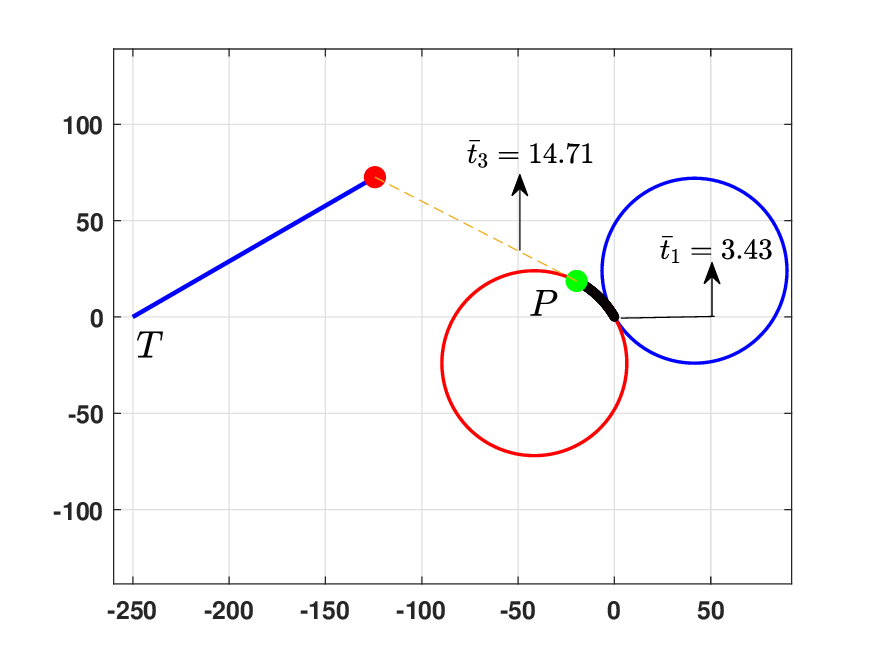} \\ \hspace{-2mm}
{\footnotesize (b) Optimal path for equal velocities of \\pursuer and target.}\hspace{-2mm}
\end{center}
\end{minipage}
\\\hspace{-2mm}
\hspace*{-1cm}
\begin{minipage}{95mm}
\begin{center}
\hspace*{0cm}
\includegraphics[width=95mm]{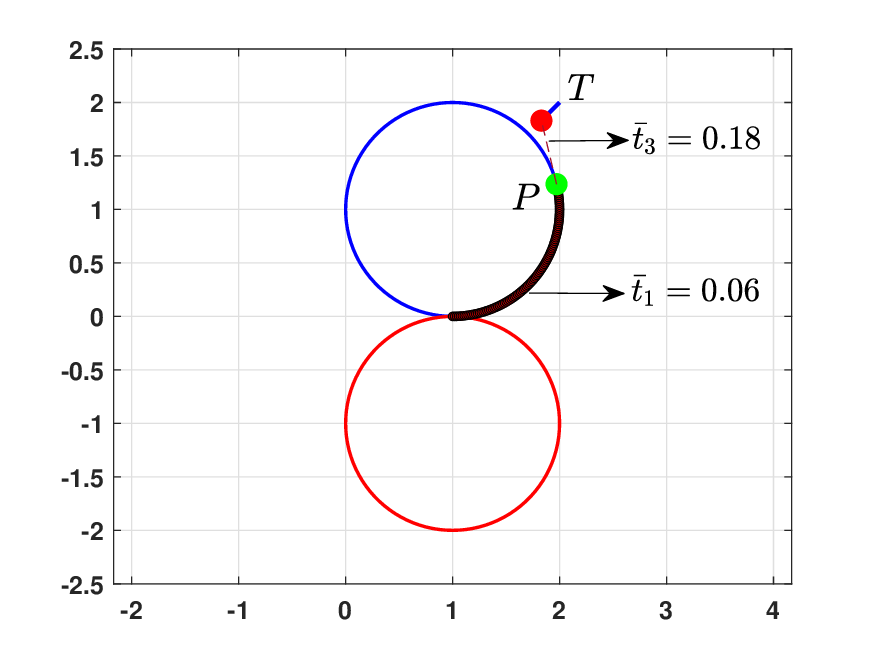} \\ \hspace{-2mm}
{\footnotesize (c) Optimal path for the example shown \\ in Figure 4(b) of \cite{Looker2008}.}
\end{center}
\end{minipage}
\vspace{.5cm}
\caption{\sf\small Optimizing the pursuer's path to intercept the target under various parameter combinations. }
\label{comparison}
\end{figure}


\newpage
\section{Conclusion}

\noindent In this paper, we introduced a mathematical model and an efficient algorithm to determine the shortest path for a pursuer intercepting a moving target with constant speed. Considering Dubin's path concepts, the proposed model combines circular curves and straight-line segments to minimize the pursuer's path length or interception time. We provided a geometric interpretation of the model and formulated it mathematically to identify optimal interception paths. The model's constraints were discussed through propositions, and an algorithm was developed to solve it. Extensive numerical experiments were conducted, considering diverse positions and heading angles for the pursuer and target, with results presented in tables and figures. Additionally, we compared the proposed model with an existing approach from \cite{Looker2008}, demonstrating its superior performance. The model exhibited a high convergence rate, efficiently approximating solutions across a wide range of parameters while requiring less computational time.

\noindent \textbf{Limitations and Future works.} While our study has made significant progress, there remain intriguing avenues for further exploration in this field. Several questions remain open for future research: One can investigate 
\begin{itemize}
    \item How does the model perform when the target's speed and direction change dynamically? Adapting the model for dynamic targets could enhance its applicability in real-world settings.
    \item How can the model be extended to incorporate obstacle avoidance strategies? Introducing fixed obstacles in the pursuit scenario could present complex challenges that warrant exploration.
    \item How does the model behave in scenarios involving multiple pursuers and targets? Investigating the interaction and coordination of multiple objects in pursuit situations could provide valuable insights.
\end{itemize}
\noindent Addressing the above open questions can further enhance the adaptability and applicability of our proposed model, paving the way for more sophisticated and effective pursuit strategies in diverse real-world scenarios.

\newpage
\noindent \textbf{\large {Acknowledgments:}} 

\noindent Moreover, we acknowledge Chatgpt(3.5) AI tools which is used to rephrase, edit and polish the authors written text for spelling, grammar, or general style.


\section{Statements $\&$ Declarations}
\textbf{Funding and/or Conflicts of interests/Competing interests:}

\noindent Masuda Akter would like to express her gratitude to Ministry of Science $\&$ Technology, Bangladesh for providing financial help in the form of NST fellowship with Reference no. 120005100-3821117, Reg. no. 6 $\&$ Session: 2022-2023.\par

\noindent Author Dr Mohammed Mustafa Rizvi and Dr Mohammad Forkan have no relevant financial or non-financial interests to disclose. Masuda Akter received grants to complete her research degree M Phil and the manuscript is part of her research work.

\noindent \textbf{Author Contributions:}
All authors contributed to the study conception and design. Material preparation, data collection and analysis were performed by Masuda Akter, Dr Mohammed Mustafa Rizvi and Dr Mohammad Forkan. All authors have the same contribution to prepare the first draft of the manuscript. All authors read and approved the final manuscript.” 
\end{document}